\documentclass[11pt,twoside]{amsart}
\usepackage[latin1]{inputenc}
\usepackage[T1]{fontenc}
\usepackage{mathtools}
\usepackage{graphicx}
\usepackage{tikz}
\usetikzlibrary{chains}

\usepackage[latin1]{inputenc}
\usepackage{amsmath}
\usepackage{amsthm}
\usepackage{amssymb}
\usepackage[all]{xy}
\usepackage{hyperref}
\newtheorem{thm}{Theorem}[section]

\newtheorem{prop}[thm]{Proposition}

\newtheorem{cor}[thm]{Corollary}
\newtheorem{conj}[thm]{Conjecture}

\numberwithin{equation}{section}

\theoremstyle{definition}

\newtheorem{remark}[thm]{Remark}

\newcommand{\qqed}{\hspace*{\fill}$\Box$}

\newcommand{\Db}{{\rm D}^{\rm b}}
\newcommand{\Mot}{{\rm Mot}}

\newcommand{\Br}{{\rm Br}}
\newcommand{\CH}{{\rm CH}}
\newcommand{\NS}{{\rm NS}}

\newcommand{\rk}{{\rm rk}}

\newcommand{\Hom}{{\rm Hom}}

\newcommand{\ch}{{\rm ch}}
\newcommand{\td}{{\rm td}}

\newcommand{\cal}{\mathcal}

\newcommand{\ke}{{\cal E}}

\newcommand{\ko}{{\cal O}}

\newcommand{\LL}{\mathbb{L}}
\newcommand{\ZZ}{\mathbb{Z}}
\newcommand{\QQ}{\mathbb{Q}}

\newcommand{\CC}{\mathbb{C}}

\newcommand{\hh}{\mathfrak{h}}

\renewcommand{\to}{\xymatrix@1@=15pt{\ar[r]&}}
\renewcommand{\rightarrow}{\xymatrix@1@=15pt{\ar[r]&}}
\renewcommand{\leftarrow}{\xymatrix@1@=15pt{&\ar[l]}}
\renewcommand{\mapsto}{\xymatrix@1@=15pt{\ar@{|->}[r]&}}
\renewcommand{\twoheadrightarrow}{\xymatrix@1@=18pt{\ar@{->>}[r]&}}
\renewcommand{\hookrightarrow}{\xymatrix@1@=15pt{\ar@{^(->}[r]&}}
\newcommand{\hook}{\xymatrix@1@=15pt{\ar@{^(->}[r]&}}
\newcommand{\congpf}{\xymatrix@1@=15pt{\ar[r]^-\sim&}}
\renewcommand{\cong}{\simeq}

\begin{document}

\title[]{Motives of isogenous K3 surfaces}

\author[D.\ Huybrechts]{D.\ Huybrechts}

\address{Mathematisches Institut,
Universit{\"a}t Bonn, Endenicher Allee 60, 53115 Bonn, Germany}
\email{huybrech@math.uni-bonn.de}

\begin{abstract} \noindent
We  prove that isogenous K3 surfaces have isomorphic Chow motives.
This provides a motivic interpretation of a long standing conjecture of {\v{S}}afarevi{\v{c}} which has been settled
only recently by Buskin. The main step consists of a new proof of {\v{S}}afarevi{\v{c}}'s conjecture
that circumvents the analytic parts in \cite{Buskin}, avoiding twistor spaces and non-algebraic K3 surfaces.
 \vspace{-2mm}
\end{abstract}

\maketitle

{\let\thefootnote\relax\footnotetext{The author is supported by the SFB/TR 45 `Periods,
Moduli Spaces and Arithmetic of Algebraic Varieties' of the DFG
(German Research Foundation).}
\marginpar{}
}

Two complex projective K3 surface $S$ and $S'$ are called \emph{isogenous} if there exists a Hodge isometry
$\varphi\colon H^2(S,\QQ)\congpf H^2(S',\QQ)$, i.e.\ an isomorphism of $\QQ$-vector spaces compatible with the
intersection pairing as well as the Hodge structure on both sides. Via Poincar\'e duality and K\"unneth formula,
$\varphi$ corresponds to a Hodge class $[\varphi]\in H^{2,2}(S\times S',\QQ)$ on the product $S\times S'$ of the two surfaces.

In \cite{ShafICM} {\v{S}}afarevi{\v{c}} asked whether any such $[\varphi]$ is 
algebraic, i.e.\  of the form  $[\varphi]=\sum n_i[Z_i]$ for certain surfaces $Z_i\subset S\times S'$ and rational numbers $n_i$.
Forty years later this  was answered affirmatively by Buskin \cite{Buskin}. The result  confirms the Hodge conjecture
in a geometrically interesting situation and can be viewed as a generalization of the global Torelli theorem for K3 surfaces.

Indeed, the global Torelli theorem for K3 surfaces asserts that
any effective integral Hodge isometry $\varphi\colon H^2(S,\ZZ)\congpf H^2(S',\ZZ)$
can be lifted to an isomorphism $f\colon S\congpf S'$ and so $[\varphi]=[\Gamma_f]$, which is algebraic.
Note that 
the global Torelli theorem not only answers {\v{S}}afarevi{\v{c}}'s question for (effective) integral
Hodge isometries,  it also provides a motivic reason for the class $[\varphi]$ being algebraic,  namely that it is
induced by an isomorphism between $S$ and $S'$.

\smallskip

Examples of rational Hodge isometries can be produced by means of moduli spaces of sheaves,
often leading to non-isomorphic but isogenous K3 surfaces. Assume
 $S'=M(v)$ is a fine moduli space of stable sheaves  on $S$.
 Then the universal family $\ke$ on $S\times S'$, an analogue of the Poincar\'e bundle for
abelian varieties, provides a class $\ch_2(\ke)\in H^{2,2}(S\times S',\QQ)$.
As shown by Mukai  \cite{MukaiVB}, a minor modification of this class yields indeed
a Hodge isometry $H^2(S,\QQ)\cong H^2(S',\QQ)$. In fact, it defines an integral Hodge isometry
$T(S)\cong T(S')$ between the transcendental lattices of the two surfaces and  a rational isometry $\NS(S)\otimes\QQ\cong \NS(S')\otimes \QQ$ between
their algebraic parts. The motivic nature of the rational Hodge isometry,
beyond being induced by a universal sheaf, has been explained in \cite{HuyM}:
For any fine moduli space $S'=M(v)$, the induced Hodge isometry $H^2(S,\QQ)\cong H^2(S',\QQ)$
can be lifted to an isomorphism  $\hh(S)\cong\hh(S')$ between the Chow motives of $S$ and $S'$. 

 Mukai also constructs in \cite{MukaiVB} further classes that yield  non-integral Hodge isometries between 
the transcendental parts  $T(S)\otimes\QQ\cong T(S')\otimes\QQ$ by allowing coarse moduli spaces, i.e.\ moduli spaces for 
which only a quasi-universal or a twisted universal family $\ke$ exists.  This approach  has led to the  verification of {\v{S}}afarevi{\v{c}}'s
conjecture for Picard rank $\rho(S)\geq 5$,
see \cite{MukaiVB,Nikulin} and Remark \ref{rem:History}.

\medskip

Our first main result provides a moduli  interpretation of isogenies between K3 surfaces:

\begin{thm}\label{thm:main1} Any Hodge isometry $H^2(S,\QQ)\cong H^2(S',\QQ)$ between two
complex projective K3 surfaces can be written as a composition of Hodge isometries between projective K3 surfaces
$$H^2(S=S_0,\QQ)\cong H^2(S_1,\QQ)\cong\cdots\cong H^2(S_{n-1},\QQ)\cong H^2(S_n=S',\QQ),$$
with $S_i$ isomorphic to a coarse moduli space of complexes of twisted coherent sheaves on $S_{i-1}$ and
the Hodge isometry $H^2(S_{i-1},\QQ)\cong H^2(S_i,\QQ)$ induced (up to sign) by a twisted universal
family $\ke$ of complexes of twisted sheaves on $S_{i-1}$.
\end{thm}

In the language of derived categories, the result says that there exist Brauer classes
$\alpha\in\Br(S)$, $\alpha_1,\beta_1\in\Br(S),\ldots,\alpha_{n-1},\beta_{n-1}\in\Br(S_{n-1})$, and $\alpha'\in\Br(S')$
and exact linear  equivalences between bounded derived categories of twisted coherent sheaves
\begin{equation}\label{eqn:der}
\begin{array}{rlll}
\Db(S,\alpha)\cong&\!\!\!\Db(S_1,\alpha_1),&&\\
                              &\!\!\!\Db(S_1,\beta_1)\cong&\!\!\!\Db(S_2,\alpha_2) ,&\\
                              &&\,\,\,\vdots&\\
                              &&\!\!\!\Db(S_{n-2},\beta_{n-2})\cong&\!\!\!\Db(S_{n-1},\alpha_{n-1}),\\
                              &&&\!\!\!\Db(S_{n-1},\beta_{n-1})\cong\Db(S',\alpha').
\end{array}
\end{equation}
This usually does not mean that $\Db(S,\alpha)$ and $\Db(S',\alpha')$ are equivalent for appropriated choices
of $\alpha\in\Br(S)$ and $\alpha'\in\Br(S')$, see Remark \ref{rem:twisted}. It should not be too difficult to improve Theorem \ref{thm:main1} such that
the $S_i$ are moduli spaces of twisted sheaves (and not complexes of those).

\smallskip

Combining Theorem \ref{thm:main1} with the arguments in  \cite{HuyM} generalized to the twisted case, one deduces a
motivic interpretation of the notion of isogeneous K3 surfaces:

\begin{thm}[Motivic {\v{S}}afarevi{\v{c}} conjecture]\label{thm:main2} 
Any Hodge isometry $H^2(S,\QQ)\cong H^2(S',\QQ)$ between two complex projective K3 surfaces
can be lifted to an isomorphism of Chow motives $\hh(S)\cong\hh(S')$. In particular, two
isogenous K3 surfaces have isomorphic Chow motives:
$$H^2(S,\QQ)\cong H^2(S',\QQ)\text{\rm ~(Hodge isometry)} \Rightarrow \hh(S)\cong\hh(S').$$
\end{thm}

Note that by Witt's theorem, there exists a Hodge isometry $H^2(S,\QQ)\cong H^2(S',\QQ)$ if and only
if there exists a Hodge isometry $T(S)\otimes\QQ\cong T(S')\otimes\QQ$. For integral coefficients this fails,
which results in two global Torelli theorems, the classical  and the derived,\footnote{
(i) $H^2(S,\ZZ)\cong H^2(S',\ZZ) \text{\rm ~(Hodge isometry)}  \Leftrightarrow S\cong S'$ (isomorphism);\\
\phantom{HGi}(ii) $T(S)\cong T(S') \text{\rm ~(Hodge isometry)} \Leftrightarrow \Db(S)\cong\Db(S')$ (exact linear equivalence).}
  see \cite{HuyFM,HuySeat,HuyK3} for
references.

\medskip

The following strengthening of Theorem \ref{thm:main2} is expected. It relaxes the assumption from 
the existence of a Hodge isometry to the existence of a simple isomorphism
of Hodge structures, so one that is not necessarily compatible with the intersection pairing (cf.\ Section \ref{sec:comments}):

\begin{conj}[Motivic global Torelli theorem]\label{conj}
For two complex projective K3 surfaces $S$ and $S'$ the following conditions are equivalent:\\
{\rm (i)} $H^2(S,\QQ)\cong H^2(S',\QQ)$ (isomorphism of rational Hodge structures);\\
{\rm (ii)} $\hh(S)\cong\hh(S')$ (isomorphism of  Chow
motives).
\end{conj}


Theorem \ref{thm:main1} has the following immediate consequences first proved in \cite{Buskin}, 
see Proposition \ref{prop:thm} and Remark \ref{rem:Schlic}.

\begin{cor}[Buskin]\label{cor:Buskin}
{\rm (i)} Any Hodge isometry $\varphi\colon H^2(S,\QQ)\congpf H^2(S',\QQ)$ between complex projective K3 surfaces
yields an algebraic class $[\varphi]\in H^{2,2}(S\times S',\QQ)$.

{\rm (ii)} If $S$ is a complex projective K3 surface with complex multiplication, i.e.\ ${\rm End}_{\rm Hdg}(T(S)\otimes\QQ)$ is
a CM-field, then the Hodge conjecture holds for $S\times S$, cf.\ Remark \ref{rem:Schlic}.
\end{cor}

The approach to {\v{S}}afarevi{\v{c}}'s conjecture presented here differs from the one in \cite{Buskin}. It is more algebraic in spirit, which allows
for the motivic interpretation of the conjecture as presented in Theorem \ref{thm:main2}. Central to his argument,
Buskin proves `twistor path connectedness' of the moduli space of pairs of K3 surfaces together with an isogeny between
them to reduce the situation to coarse moduli spaces of untwisted bundles. In our proof, cyclic isogenies are lifted directly
to the level of derived categories of twisted K3 surfaces \cite{HuySt}, thus avoiding analytic K3 surfaces and
global moduli considerations. The notion of Hodge structures of twisted K3 surfaces introduced in \cite{HuyCY,HuySt1}
provides an efficient tool to deal with the lattice theoretic parts and, in particular, replaces Buskin's $\kappa$-classes.

\vskip0.6cm

\noindent
{\bf Acknowledgements:} I would like to heartily thank the participants of the inspiring IC Geo\-metry Seminar on 
Buskin's paper and especially Lenny Taelman for his energy in organizing it. I am grateful to Fran\c{c}ois Charles,
Lenny Taelman, and Andrey Soldatenkov for comments on a first version of this paper.
Many thanks to Rahul Pandharipande for an invitation to the ETH, Zurich, where the main idea took shape.
Hospitality and financial support of the  Erwin Schr\"odinger Institute, where the first version of this paper was
completed, is gratefully acknowledged.

%
%
%

\section{Derived equivalence of isogenous K3 surfaces}

This section is the technical heart of the paper.  We show how to lift rational Hodge isometries to
exact linear equivalences between bounded derived categories of twisted sheaves and use this
to prove {\v{S}}afarevi{\v{c}}'s conjecture. The first reduction step to cyclic Hodge isometries is taken
from \cite{Buskin}. The rest of the argument uses  twisted Chern characters and the main result of \cite{HuySt},
instead of $\kappa$-classes and twistor space deformations. A brief comparison of the two approaches is
included.

\subsection{} As in \cite{Buskin}, we apply the classical Cartan--Dieudonn\'e theorem to reduce
{\v{S}}afarevi{\v{c}}'s conjecture to an easier case. Recall that for any lattice $\Lambda$ and any rational isometry $\varphi\colon\Lambda_\QQ\congpf
\Lambda_\QQ$, there exist  $b_i\in\Lambda_\QQ$, $i=1,\ldots,k$, with $(b_i)^2\ne0$, such that
$\varphi$ equals the composition $$\varphi=s_{b_1}\circ\cdots\circ s_{b_k}$$ of reflections $s_{b_i}\colon
x\mapsto x-\frac{2(x.b_i)}{(b_i)^2}b_i$. Note that the number of reflections can be bounded by
$k\leq \rk\,\Lambda$. Clearly, we may assume that all $b_i\in\Lambda_\QQ$ are contained in the lattice $\Lambda$ and that they are actually primitive elements of $\Lambda$.\footnote{Buskin in \cite{Buskin} only uses the property of a reflection $\varphi=s_\delta$ to be \emph{cyclic},
i.e.\ to have the property that $\varphi^{-1}(\Lambda)\cap\Lambda \subset\Lambda$ and $\varphi(\Lambda)\cap \Lambda\subset\Lambda$ have cyclic quotients. We shall really have to work with reflections.}


Combining this with the surjectivity of the period map, one finds that any Hodge isometry $H^2(S,\QQ)\congpf H^2(S',\QQ)$
can be written as a composition of Hodge isometries
\begin{equation}\label{eqn:compcycl}
H^2(S=S_0,\QQ)\congpf H^2(S_{1},\QQ)\congpf \cdots \congpf H^2(S_n=S',\QQ),
\end{equation}
such that after choosing markings $\Lambda\cong H^2(S_i,\ZZ)$ and $\Lambda\cong H^2(S_{i+1},\ZZ)$
the Hodge isometry $H^2(S_i,\QQ)\congpf H^2(S_{i+1},\QQ)$ is of the form
$s_{b_i}$. We call a Hodge isometry of this type \emph{reflective}.
Thus, Theorem \ref{thm:main1} is a consequence of the following result which will be proved in this section.

\begin{thm}\label{thm:main1bis}
Assume $\varphi\colon H^2(S,\QQ)\congpf H^2(S',\QQ)$ is a reflective Hodge isometry. Then $S'$ is a coarse moduli space of complexes
of twisted coherent sheaves on $S$ and $\varphi$ is (up to sign)  induced by a twisted universal family of complexes of twisted sheaves.
\end{thm}

In other words, we claim that $\varphi$ (up to sign) is induced by the Fourier--Mukai kernel $\ke$ of an exact linear equivalence
$\Phi_\ke\colon\Db(S,\alpha)\congpf\Db(S',\alpha')$
for suitable Brauer classes $\alpha\in \Br(S)$ and $\alpha'\in \Br(S')$. Here, $\ke$ is an object in the bounded
derived category  $\Db(S\times S',\alpha^{-1}\boxtimes\alpha')$ of $\alpha^{-1}\boxtimes\alpha'$-twisted coherent sheaves on $S\times S'$,
see below for details on the action on cohomology.

\subsection{}  We begin with a few explicit lattice computations. 
Let $\varphi=s_b\colon \Lambda_\QQ\congpf\Lambda_\QQ$ be a reflection with $b\in\Lambda$ primitive. Then, for $x\in \Lambda$, the image $\varphi(x)\in\Lambda_\QQ$ is contained in $\Lambda$ if and only if $(x.b)$ is divisible by $n\coloneqq (b)^2/2$.
So, if we let $B\coloneqq \frac{1}{n}b\in\Lambda_\QQ$, then $\varphi$ induces an isometry
of $\Lambda_B\coloneqq \{x\in\Lambda\mid (x.B)\in \ZZ\}\subset\Lambda$. This is a finite index sublattice with a cyclic quotient of order
$n$. Note that $\varphi(z)=z-(z.B)b$ and, hence, $\varphi(B)=-B$. Next consider 
\begin{equation}\label{primemb}
\exp(B)\colon \xymatrix{\Lambda_B\ar@{^(->}[r]&\widetilde\Lambda}\coloneqq\Lambda\oplus U,~ÊÊx\mapsto x+(B.x)f,
\end{equation}
which is a primitive embedding of lattices.
Here, $U$ is the hyperbolic plane with the standard isotropic basis $e,f$ with $(e.f)=-1$. The sign is inserted to make
$U$ naturally isomorphic to $H^0(S,\ZZ)\oplus H^4(S,\ZZ)\cong\ZZ\cdot e\oplus\ZZ\cdot f$ endowed with the Mukai pairing.  The orthogonal complement
$(\exp(B)(\Lambda_B))^\perp\subset\widetilde \Lambda$ is the lattice spanned by the isotropic vectors
$b+ne+f$ and $-f$, which is thus isomorphic to the twisted hyperbolic plane $U(n)$. The
isometry $\varphi\colon\Lambda_B\congpf\Lambda_B$ extends to an
isometry $\tilde\varphi\colon\widetilde\Lambda\congpf\widetilde\Lambda$, i.e.\ there exists a commutative diagram of the form
\begin{equation}\label{extphi}
\xymatrix{\Lambda_B\ar[d]^\cong_\varphi\ar@{^(->}[rr]^{\exp(B)}&&\widetilde\Lambda\ar[d]_{\tilde\varphi}^\cong\\
\Lambda_B\ar@{^(->}[rr]_{\exp(B)}&&\widetilde\Lambda.}
\end{equation}
The  extension can be given explicitly as
$$\tilde\varphi(z+re+sf)\coloneqq\varphi(z)+((B.z)-s)b+n((B.z)-(r/n)-s)e-sf.$$ The compatibility
with $\varphi$ is easily shown using
$(B.z)=-(B.\varphi(z))$. On $(\exp(B)(\Lambda_B))^\perp \cong U(n)$, $\tilde\varphi$ interchanges the two basis vectors
$b+ne+f$ and $-f$. This shows that $\tilde\varphi$ is indeed an isometry. An explicit computation shows that
indeed $\tilde\varphi(\widetilde\Lambda)=\widetilde\Lambda$.
\smallskip

Let us now apply this to a reflective Hodge isometry $\varphi\colon H^2(S,\QQ)\congpf H^2(S',\QQ)$.
The analogue of $b\in \Lambda$ and $B\in\Lambda_\QQ$ in the above setting are now classes $b\in H^2(S,\ZZ)$ and $B=(1/n)b\in H^2(S,\QQ)$.
We set $b'\coloneqq-\varphi(b)\in H^2(S',\ZZ)$ and $B'\coloneqq -\varphi(B)\in H^2(S',\QQ)$.
Then the Hodge isometry
$\varphi$ of rational Hodge structures induces a Hodge isometry of integral Hodge structures
$$\varphi\colon H^2(S,\ZZ)_B\coloneqq\{x\in H^2(S,\ZZ)\mid (x.B)\in\ZZ\}\congpf H(S',\ZZ)_{B'}\coloneqq\{x'\in H^2(S',\ZZ)\mid (x'.B')\in\ZZ\}.$$
Furthermore, (\ref{primemb}) becomes the primitive embedding of lattices
$$\exp(B)\colon\xymatrix{H^2(S,\ZZ)_B\ar@{^(->}[r]& \widetilde H(S,\ZZ),}~x\mapsto
x+x\wedge B,$$where $\widetilde H(S,\ZZ)$ is the Mukai lattice, i.e.\ the lattice $H^*(S,\ZZ)$ with a sign change in the pairing of $H^0$ and $H^4$. The analogue of (\ref{extphi}) is the commutative diagram
\begin{equation}\label{extphi2}
\xymatrix{H^2(S,\ZZ)_B\ar[d]^\cong_\varphi\ar@{^(->}[rr]^{\exp(B)}&&\widetilde H(S,\ZZ)\ar[d]_{\tilde\varphi}^\cong\\
H^2(S',\ZZ)_{B'}\ar@{^(->}[rr]_{\exp(B')}&&\widetilde H(S',\ZZ).}
\end{equation}
with $\tilde\varphi(r,z,s)=\left(n((B.z)-(r/n)-s),\varphi(z)+((B.z)-s)b',-s\right)$.

The Hodge structure of $H^2(S,\ZZ)_B$, inherited from $H^2(S,\ZZ)$, induces a natural Hodge structure of weight two
on the Mukai lattice $\widetilde H(S,\ZZ)$. The lattice $\widetilde H(S,\ZZ)$ endowed with this Hodge structure is
denoted $\widetilde H(S,B,\ZZ)$. Explicitly, the $(2,0)$-part of $\widetilde H(S,B,\ZZ)$ is
spanned by $\sigma+\sigma\wedge B\in H^{2}(S,\CC)\oplus H^4(S,\CC)$ for any $0\ne \sigma\in H^{2,0}(S)\subset H^2(S,\CC)$
and the orthogonal complement $(\exp(B)(H^2(S,\ZZ)_B))^\perp\subset\widetilde H(S,B,\ZZ)$ is of type $(1,1)$. With the analogous convention for $S'$, the isometry $\tilde\varphi$ can be viewed as a Hodge isometry\begin{equation}\label{eqn:Hodgeisoready}
\tilde\varphi\colon \widetilde H(S,B,\ZZ)\congpf \widetilde H(S',B',\ZZ)
\end{equation}
that commutes with $\varphi$ via $\exp(B)$ and $\exp(B')$.

%

If $\tilde\varphi$ does not preserve the natural orientation of the
four positive directions in the  Mukai lattice, then compose $\tilde \varphi$ with the Hodge isometry
$${\rm id}_{H^0}\oplus( -{\rm id}_{H^2})\oplus {\rm id}_{H^4}\colon \widetilde H(S',B',\ZZ)\congpf \widetilde H(S',-B',\ZZ).$$
This amounts to changing $\varphi$ by a sign which does not affect our problem.

\subsection{}\label{sec:kappa}
We are now ready to evoke the main result of \cite{HuySt} which asserts that
any orientation preserving Hodge isometry (\ref{eqn:Hodgeisoready}) can be lifted to
an exact equivalence 
\begin{equation}\label{eqn:Equi}
\Phi\colon \Db(S,\alpha)\congpf \Db(S',\alpha').
\end{equation}
Here, $\alpha\in\Br(S)$ and $\alpha'\in\Br(S')$ are the Brauer classes induced by
$B$ and $B'$ via the exponential sequence $H^2(S,\QQ)\to H^2(S,\ko_S)\to H^2(S,\ko_S^\ast)$.
The order of both classes divides $n$. However, although $H\subset H^2(S,\ZZ)$ and
$H'\subset H^2(S',\ZZ)$ are subgroups of the same index $n$, in general ${\rm ord}(\alpha)\ne{\rm ord}(\alpha')$, e.g.\
for $S'$ a non-fine moduli space of untwisted sheaves one has $\alpha=1$ and ${\rm ord}(\alpha')>1$.
Let us briefly recall what it means that `$\Phi$ lifts $\tilde\varphi$' and what it implies for $\varphi$.

One knows that any exact linear equivalence (\ref{eqn:Equi}) is of Fourier--Mukai type \cite{CanSt}, i.e.\ of the
form $\Phi\cong\Phi_\ke\colon E\mapsto p_*(q^*E\otimes\ke)$ for some $\ke$ in the bounded derived category
$\Db(S\times S',\alpha^{-1}\boxtimes\alpha')$ of $\alpha^{-1}\boxtimes\alpha'$-twisted coherent sheaves on $S\times S'$ and
$p,q$ the two projections.
The induced action $\Phi_\ke^{B,B'}\colon\widetilde H(S,B,\ZZ)\congpf \widetilde H(S',B',\ZZ)$ is the correspondence given by
the class $\ch^{-B,B'}(\ke)\cdot\sqrt{\td(S\times S')}$, where the twisted Chern
character is determined by the property $\ch^{-B,B'}(\ke)^n=\exp(-b,b')\cdot\ch(\ke^{\otimes n})$.
As $b=nB$ and $b'=nB'$ are both integral classes, $\ke^{\otimes n}$ is naturally untwisted and its Chern character
is well defined.\footnote{We refer to \cite{HuySt1,HuySt,HuySeat} and Section \ref{sec:Chern} for the technical details. For example,
one actually has to choose cocyles $b=\{b_{ijk}\}$, $B=\{B_{ijk}\coloneqq(1/n)b_{ijk}\}$, and
$\alpha=\{\alpha_{ijk}=\exp(B_{ijk})\}$ to make $\ke$ naturally untwisted.} 

The fact that $\Phi\cong\Phi_\ke$ lifts $\tilde\varphi$ by definition simply means
that $\Phi_{\ke}^{B,B'}=\tilde\varphi$ and the commutativity of  (\ref{extphi2}) becomes
$\varphi(x)=\left(\sqrt[n]{\ch(\ke^{\otimes n})}\cdot\sqrt{\td(S\times S')}\right)_*(x)$, cf.\ Section \ref{sec:Chern}.
In other words,
$$[\varphi]=\left(\sqrt[n]{\ch(\ke^{\otimes n})}\cdot\sqrt{\td(S\times S')}\right)_{(2,2)}\in H^2(S,\QQ)\otimes H^2(S',\QQ),$$
which is clearly an algebraic class.

The discussion above is summarized by the following reformulation of Theorem \ref{thm:main1bis}, also proving Corollary \ref{cor:Buskin} (i).

\begin{prop}\label{prop:thm}
Assume $\varphi\colon H^2(S,\QQ)\congpf H^2(S',\QQ)$ is a cyclic Hodge isometry. Then 
there exists an exact equivalence $\Phi\colon \Db(S,\alpha)\congpf\Db(S',\alpha')$
which induces $\varphi$ (up to sign) in the above sense. In particular, $[\varphi]$ is an algebraic class.\qqed
\end{prop}

\begin{remark}\label{rem:twisted} It seems natural to ask whether maybe
the existence of an arbitrary Hodge iso\-morphism $H^2(S,\QQ)\cong H^2(S',\QQ)$, so one that 
is not necessarily an isometry,  also implies $\Db(S,\alpha)\cong\Db(S',\alpha')$ for appropriately chosen Brauer classes
$\alpha\in\Br(S)$ and $\alpha'\in\Br(S')$. Although we have not worked out a concrete example, this seems
unlikely for two reasons: First, although any equivalence $\Db(S)\cong\Db(S')$ yields a natural isomorphism $\Br(S)\cong
\Br(S')$
, one should not expect that for a Brauer classe $\alpha\in\Br(S)$ and its image  $\alpha'\in\Br(S')$,
there always exists an equivalence $\Db(S,\alpha)\cong\Db(S',\alpha')$ (simply because for very general choices all two-dimensional
moduli spaces of objects in $\Db(S,\alpha)$ should be isomorphic to $S$).
Second, an equivalence $\Db(S,\alpha)\cong\Db(S',\alpha')$ only  induces a natural isomorphism $T(S,\alpha)\cong T(S',\alpha')$
but none between the untwisted transcendental lattices and hence none between the Brauer groups.
In particular, we do not a priori expect an arbitrary (non-cyclic)
Hodge isometry $H^2(S,\QQ)\cong H^2(S',\QQ)$ to be induced by some equivalence
$\Db(S,\alpha)\cong\Db(S',\alpha')$.
\end{remark}

So, in order to turn Theorem \ref{thm:main1} into an `if and only if'-statement, one could define $S$ and $S'$ to be 
\emph{twisted derived equivalent} if there exists a diagram as in (\ref{eqn:der}). Then one has

\begin{cor}[Twisted derived  global Torelli theorem]
Two complex projective K3 surfaces $S$ and $S'$ are isogenous if and only if they are twisted derived equivalent.\qed
\end{cor}

\subsection{} We conclude this section with a comparison to the earlier approaches by Buskin \cite{Buskin}
and Mukai \cite{MukaiVB}.

\begin{remark}\label{rem:History}
Mukai's approach in \cite{MukaiVB} was rather similar. Instead of  decomposing a given Hodge isometry
$ H^2(S,\QQ)\cong H^2(S',\QQ)$ into cyclic ones as in (\ref{eqn:compcycl}), he suggested to only
decompose the induced Hodge isometry $T(S)_\QQ\cong T(S')_\QQ$ into cyclic ones:
\begin{equation}\label{eqn:compcycltrans}
T(S)_\QQ=T(S_0)_\QQ\cong T(S_{1})_\QQ\cong\cdots\cong T(S_n)_\QQ=T(S')_\QQ.
\end{equation}
This reduces {\v{S}}afarevi{\v{c}}'s conjecture 
to a Hodge isometry $T(S)_\QQ\cong T(S')_\QQ$ for which the intersection $T\coloneqq T(S)\cap T(S')$
has finite cyclic quotients in $T(S)$ and in $T(S')$
and, using $\Br(S)\cong\Hom(T(S),\QQ/\ZZ)$, can then be written as $$T(S)\supset T(S,\alpha)\cong T\cong T(S',\alpha')\subset T(S')$$
for certain Brauer classes $\alpha\in \Br(S)$ and $\alpha'\in\Br(S')$. If, furthermore,
$T\cong T(\tilde S)$ for some K3 surface $\tilde S$, then $S$ and $S'$ can both be viewed as coarse moduli spaces
of sheaves on $\tilde S$ and the inclusions $T(S)\supset T(\tilde S)\subset T(S)$ are both algebraic, induced
by the twisted universal sheaves.
Unfortunately, the existence of the surface $\tilde S$ cannot be deduced from the surjectivity of period in general
(in contrast to (\ref{eqn:compcycl})) and, in fact,
$\tilde S$ may simply not exist. This limited
Mukai's approach \cite{MukaiVB} to the case $\rho(S)\geq 11$, later improved to $\rho(S)\geq 5$ by Nikulin  \cite{Nikulin}.

The idea of the present approach is that $\tilde S$ is not needed. Instead of viewing $S$ and $S'$ as coarse
moduli spaces of untwisted sheaves on some auxiliary K3 surface $\tilde S$, one realizes $S'$ directly as a coarse(!) moduli
space of (complexes of) twisted(!) sheaves on $S$. This accounts for the two additional Brauer classes at once: $\alpha\in\Br(S)$, as the twist
with respect to which one considers the twisted sheaves on $S$, and $\alpha'\in\Br(S')$, as the obstruction to the existence
of a universal family on $S\times S'$ (of $\alpha$-twisted sheaves in $S$).
\end{remark}

Buskin starts with the case of  a coarse moduli space of vector bundles $S'=M(v)$ with a twisted universal
bundle $\ke\in{\rm Coh}(S\times S',1\boxtimes\alpha')$, where $\alpha'$ is the obstruction to the
existence of a universal family. He then considers the graded(!) Hodge isometry $\widetilde H(S,\QQ)\congpf \widetilde H(S',\QQ)$
induced by the class $\kappa(\ke)\cdot\sqrt{\td(S\times S')}$, where $\kappa(\ke)=\sqrt[n]{\ch(\ke^{\otimes n}\otimes\det(\ke^\ast))}$
(up to dualizing $\ke$). Here the crucial observation is that $\ke^{\otimes n}\otimes\det(\ke^\ast)$ is naturally untwisted for any representative of the Brauer
classes. A straightforward computation shows
that $\kappa(\ke)$ differs from $\ch^{-B',B}(\ke)$ by the factor $\exp(-{\rm c}_1(\ke^{\otimes n})/n^3)\cdot \exp(B,-B')$
and so the difference between the action of $\kappa(\ke)$ and $\varphi$  is caused by an additonal
factor $\exp(-{\rm c}_1(\ke^{\otimes n})/n^3)$ on the product. Note that by construction $\varphi$
preserves the degree two part, which is not obvious from this comparison.

So, in the language of Remark \ref{rem:History}, the starting point in \cite{Buskin} is of the form
$T(S)\cong T\cong T(S',\alpha')\subset T(S')$. In a next step, $S$ and $S'$ are deformed along a twistor space
to K3 surfaces $S_t$ and $S'_t$.
This is a topologically  trivial operation, so yields isometries $H^2(S_t,\QQ)\cong H^2(S,\QQ)\cong H^2(S',\QQ)\cong H^2(S'_t,\QQ)$
and, for a suitable simultaneous choice of the twistor deformation, in fact a Hodge isometry $H^2(S_t,\QQ)\cong H^2(S'_t,\QQ)$. However, on the
transcendental part it leads to a situation of the form $T(S_t)\supset T(S_t,\alpha_t)\cong T_t\cong T(S'_t,\alpha_t')\subset T(S'_t)$, which
provides more flexibility.
Then Buskin argues that although $T_t$ may not be the transcendental part of a K3 surface, the correspondence
is still algebraic. Indeed, the (partially) twisted bundle $\ke$ deforms to a (completely) twisted
bundle $\ke_t$ on $S_t\times S'_t$, which uses the existence of Hermite--Einstein metrics on stable bundles.
At this point it becomes important to work not with complexes of sheaves as in our approach but with vector bundles.

To conclude, Buskin has to show that any cyclic Hodge isometry $H^2(\tilde S,\QQ)\cong H^2(\tilde S',\QQ)$ can
be reached by this procedure, applying several twistor deformations
which requires to work with non-projective K3 surfaces.

It should be possible to build upon Buskin's work to prove Proposition \ref{prop:thm}. The deformation of $\ke$ to $\ke_t$,
along several twistor lines and involving non-projective K3 surface when changing from one to the next twistor line, should
yield an equivalence. The approach presented here is more direct and more suitable to deal with K3 surfaces over other fields.

\section{Motives of coarse moduli spaces of twisted sheaves}
 
In this section we show the following result which generalizes \cite{HuyM} from the case of fine moduli spaces
of (complexes of) untwisted sheaves  to the case of coarse(!) moduli spaces of (complexes of) twisted(!) sheaves.

\begin{thm}\label{thm:mainsreform}
Any exact linear equivalence $\Db(S,\alpha)\cong\Db(S',\alpha')$
between twisted projective K3 surfaces $(S,\alpha)$ and $(S',\alpha')$ over an arbitrary field $k$
induces an isomorphism between their Chow motives
$$\hh(S)\cong\hh(S').$$
\end{thm}

\subsection{}\label{sec:Chern}
We shall need a few facts on Chern character of twisted sheaves. The arguments are all standard, but as 
there is no appropriate reference we sketch the relevant bits in a rather ad hoc manner.

Let $\alpha\in\Br(X)$ be a Brauer class on a smooth projective variety $X$  with a \v{C}ech representative (in the analytic or \'etale topology) $\alpha=\{\alpha_{ijk}\in\ko^\ast(U_{ijk})\}$.
We shall assume that $\alpha_{ijk}^n=1$, which is stronger than just assuming $\alpha^n=1$. 

The abelian category of $\alpha$-twisted coherent sheaves is incarnated by the category
 of $\{\alpha_{ijk}\}$-twisted coherent sheaves ${\rm Coh}(X,\{\alpha_{ijk}\})$, but we will use ${\rm Coh}(X,\alpha)$ as a shorthand
 (see \cite{HuySt} for comments on the dependence of the choice).
Now, observe that for any locally free
$\{\alpha_{ijk}\}$-twisted sheaf $E=\{E_i,\varphi_{ij}\}$ the tensor product $E^{\otimes n}=\{E_i^{\otimes n},\varphi_{ij}^{\otimes n}\}$
is naturally untwisted, i.e.\ $\varphi_{ij}^{\otimes n}\circ \varphi_{jk}^{\otimes n}\circ
\varphi_{ki}^{\otimes n}=\alpha_{ijk}^n=1$, so that Chern classes of $E^{\otimes n}$ are well defined in $\CH^*(X)$  (or in cohomology).
Now define $$\ch(E)\coloneqq\sqrt[n]{\ch(E^{\otimes n})}\in \CH^*(X)_\QQ.$$
The $n$-th root is obtained by the usual purely formal operation, using that
$\rk(E^{\otimes n})\ne0$.\footnote{For K3 surfaces with a rational point,
Chern characters of untwisted sheaves are integral. This does not hold for twisted sheaves, as taking the $n$-th root requires 
to work with rational coefficients.}

\smallskip

We leave it to the reader to check the following facts:
\begin{itemize}
\item[(i)] The definition is independent of $n$ in the sense that $\ch(E)=\sqrt[n]{\ch(E^{\otimes n})}=\sqrt[mn]{\ch(E^{\otimes mn})}$.

\item[(ii)] For locally free $\{\alpha_{ijk}\}$-twisted
sheaves $E$ and $F$ we have $\ch(E\otimes F)=\ch(E)\cdot\ch(F)$ and $\ch(E\oplus F)=\ch(E)+\ch(F)$. A similar formula holds
for exact sequences.

\item[(iii)] As $X$ is smooth projective, any twisted
sheaf admits a locally free resolution and, hence,  the Chern character is well-defined for all $E\in{\rm Coh}(X,\alpha)$
and even for objects in the bounded derived category $\Db(X,\alpha)$.\footnote{This also explains how to interprete  $\sqrt[n]{\ch(\ke^{\otimes})}$ in Section \ref{sec:kappa}.} 

\item[(iv)] For a morphism $f\colon Y\to X$ of smooth projective varieties,  the Grothendieck--Riemann--Roch formula
holds: $\ch\left(Rf_*\ch(E)\right)\cdot\td(X)=f_*\left(\ch(E)\cdot\td(Y)\right)$ in $\CH^*(X)_\QQ$ for any $E\in\Db(Y,f^*\alpha)$.
(This is easily reduced to the usual formula by tensoring both sides with $f^*G$ and $G$, respectively, for some locally free
$\{\alpha_{ijk}^{-1}\}$-twisted sheaf $G$ on $X$.)
\end{itemize}

\smallskip

Once these facts are established, the yoga of Fourier--Mukai kernels 
$\ke$, their action on the Chow ring, induced by $v(\ke)\coloneqq\ch(\ke)\sqrt{\td(S\times S')}\in\CH^*(S\times S')_\QQ$,
and how they behave under convolutions,
works as in the untwisted case, cf.\ \cite{HuyFM}. The next result is an example.
For this, we assume that $\alpha=\{\alpha_{ijk}\}$ and $\alpha'=\{\alpha_{ijk}'\}$ are Brauer classes
on K3 surfaces $S$ and $S'$, respectively, both satisfying $\alpha_{ijk}^n=1$ and $\alpha'^n_{ijk}=1$.

\begin{cor} 
Let $\Phi_\ke\colon\Db(S,\alpha)\congpf\Db(S',\alpha')$ be an exact equivalence with Fourier--Mukai kernel
$\ke\in\Db(S\times S',\alpha^{-1}\boxtimes\alpha')$. Then the induced action
\begin{equation}\label{eqn:isoChow}
v(\ke)_*\colon\CH^*(S)_\QQ\congpf\CH^*(S')_\QQ
\end{equation}
is an isomorphism of ungraded $\QQ$-vector spaces.\qed
\end{cor}

\subsection{} The rest of the argument to prove Theorem \ref{thm:mainsreform} can be copied from \cite{HuyM}.
Here is a  rough outline: First, the motive of a K3 surface is decomposed into its
algebraic and its transcendental part $\hh(S)\cong\hh^2_{\rm tr}(S)\oplus\hh_{\rm alg}(S)$, where
$\hh_{\rm alg}(S)\cong\LL^0\oplus\LL^{\oplus \rho(S)}\oplus\LL^2$ and the transcendental part
$\hh^2_{\rm tr}(S)$, introduced in \cite{KMP}, has the property that $\CH^*(\hh^2_{\rm tr}(S))=\CH^2(S)_0\otimes \QQ$.
Now, derived equivalent K3 surfaces $(S,\alpha)$ and $(S',\alpha')$ have clearly the same Picard number
and, therefore, $\hh_{\rm alg}(S)\cong\hh_{\rm alg}(S')$. Thus, it remains to find an isomorphism
$\hh^2_{\rm tr}(S)\cong\hh^2_{\rm tr}(S')$. As morphisms $\hh_{\rm tr}^2(S)\to\hh_{\rm tr}^2(S')$ in $\Mot(k)$
require degree two classes on $S\times S'$, instead of $v(\ke)\in\CH^*(S\times S')_\QQ$, which induces 
the isomorphism (\ref{eqn:isoChow}), one has to consider the degree two component $v_2(\ke)\in\CH^2(S\times S')_\QQ$.
The induced action on $\CH^*(\hh^2_{\rm tr}(S))=\CH^2(S)_0\otimes \QQ$ coincides with the action of the full Mukai vector
$v(\ke)$. Hence, 
\begin{equation}\label{eqn:isoMot}
v_2(\ke)_*\colon\hh^2_{\rm tr}(S)\to \hh^2_{\rm tr}(S')
\end{equation}
 induces isomorphisms between the Chow groups
of the motives. As this holds true after any base change, a version of Manin's identity principle then implies that (\ref{eqn:isoMot})
is an isomorphism, for details see \cite{HuyM}.

%
%

\section{Further comments}\label{sec:comments}

Let us briefly indicate the evidence for the motivic global Torelli theorem as formulated
in Conjecture \ref{conj}. According to the following remarks, Theorem \ref{thm:main2}, which provides
evidence for the equivalence of (i) and (ii) in Conjecture \ref{conj}, may also be seen as 
evidence for a much more general set of conservativity conjectures. 
Note that the following arguments apply to arbitrary surfaces (with trivial irregularity).

\begin{prop}
Assume the Hodge conjecture holds for the product  $S\times S'$ of two complex projective K3 surfaces and assume that
the motives $\hh(S)$ and $\hh(S')$ of both surfaces are Kimura finite-dimensional. Then any isomorphism
of Hodge structures $H^2(S,\QQ)\congpf H^2(S',\QQ)$ lifts to an isomorphism of motives $\hh(S)\cong\hh(S')$.
\end{prop}

\begin{proof}
The argument is similar to the proof of \cite[Thm.\ 21]{DPPed}.
 If the Hodge conjecture is assumed, the class $[\varphi]\in H^{2,2}(S\times S',\QQ)$ of any
isomorphism of Hodge structures $\varphi\colon H^2(S,\QQ)\congpf H^2(S',\QQ)$ is induced by 
a class $\gamma\in\CH^2(S\times S')_\QQ$, which defines a morphism $\gamma_*\colon\hh^2(S)\to \hh^2(S')$ of Chow motives.
As Kimura finite-dimensionality implies conservativity, cf.\ \cite[Cor.\ 3.16]{AndreBourb}, $\gamma_*$ is an isomorphism
if and only if its numerical realization, which is nothing but $\varphi$, is an isomorphism.  \end{proof}

\begin{cor}
The two conditions {\rm (i)} and {\rm (ii)} in Conjecture \ref{conj} are equivalent  if the Hodge conjecture for $S\times S'$
and Kimura's finite-dimensionality conjecture for $S$ and $S'$ hold true.\qed
\end{cor}

In an earlier version of this paper, Conjecture \ref{conj} included a third
statement about  the classes of $[S]$ and $[S']$ being equal in an appropriate localization of the
Grothendieck ring of varieties $K_0({\rm Var}(\CC))$.
For example, if $S$ and $S'$  are isogenous, then according to (\ref{eqn:der}), they are linked via a sequence of equivalences $\Db(S,\alpha)\cong \Db(S_1,\alpha_1),\ldots, \Db(S_{n-1},\beta_{n-1})\cong\Db(S',\alpha')$. We then speculated that maybe \cite[Conj.\ 1.6]{KS} (with evidence
provided by the examples studied in \cite{HL,IMOU,KS}) could also hold in the twisted
case, so that  $[S]-[S']$ in $K_0({\rm Var}(\CC))$ is annihilated by some power of the Lefschetz motive $\ell\coloneqq[{\mathbb A}^1]$, i.e.\ $[S]=[S']$ in $K_0({\rm Var}(\CC))[\ell^{-1}]$. 
However, as shown by Efimov  \cite{Ef}, this  is false.
There exist derived equivalent twisted(!) K3 surfaces that are not L-equivalent.

\begin{remark}\label{rem:Schlic}
According to \cite{Zarhin}, the endomorphism field ${\rm End}_{\rm Hdg}(T(S)\otimes\QQ)$ of the rational Hodge structure
$T(S)\otimes \QQ$ is either totally real or has
complex multiplication. The two cases can be distinguished by checking whether there exists of a Hodge isometry other than $\pm {\rm id}$,
see \cite[Ch.\ 3]{HuyK3}. The Hodge conjecture for K3 surfaces with real multiplication has been verified in only very few cases, see \cite{Schlick}.

In case of CM, the endomorphism field is spanned by Hodge isometries cf.\ \cite[Thm.\ 3.3.7]{HuyK3}, which is enough to prove Corollary \ref{cor:Buskin} (ii) and
can also be used to prove the Hodge conjecture for products $S\times S'$ of K3 surfaces with complex multiplication
for which there exists a Hodge isometry $T(S)_\QQ\cong T(S')_\QQ$.
\end{remark}


\end{document}